\documentclass[a4paper,12pt,openbib,reqno]{amsart}
\usepackage{amsmath,amsfonts,amssymb}
\usepackage{verbatim}
\usepackage{enumerate}
\usepackage{url}
\usepackage[colorlinks]{hyperref}
\usepackage[latin1]{inputenc}
\usepackage[english]{babel}
\usepackage{tikz}
\usepackage[margin=2.5cm]{geometry}
\usepackage{bbm}

\def\Z{\mathbb Z}
\def\C{\mathbb C}
\def\Ca{\mathcal C}
\def\Ja{\mathcal J}

\def \F{\mathbb F}

\def\Fq{\mathbb{F}_q}
\def\Fp{\mathbb{F}_p}
\def\Fqn{\mathbb{F}_{q^n}}
\def\Fqk{\mathbb{F}_{q^k}}
\def\Fpn{\mathbb{F}_{p^n}}
\def\Fpk{\mathbb{F}_{p^k}}
\DeclareMathOperator{\tr}{Tr}
\newcommand{\um}{\mathbbm{1}}

\theoremstyle{plain}
\newtheorem{theorem}{Theorem}[section]
\newtheorem{lemma}[theorem]{Lemma}
\newtheorem{definition}[theorem]{Definition}
\newtheorem{corollary}[theorem]{Corollary}

\newtheorem{conjecture}[theorem]{Conjecture}
\newtheorem{remark}[theorem]{Remark}
\newtheorem{example}[theorem]{Example}
\newtheorem{question}[theorem]{Question}

\def\qed{\hfill\hbox{$\square$}}

\theoremstyle{definition}

\author[J. Alves Oliveira]{Jos\'e Alves Oliveira}
\address{
	Departamento de Matem\'{a}tica\\
	Universidade Federal de Minas Gerais\\
	UFMG\\
	Belo Horizonte, MG\\
	31270-901\\
	Brazil\\
}
\email{jose-alvesoliveira@hotmail.com}

\title{Rational points on Cubic, Quartic and Sextic Curves over Finite Fields}

\date{\today
}

\subjclass[2000]{ }


\begin{document}
\baselineskip=1.6\baselineskip
	\begin{abstract}
	Let $\Fq$ denote the finite field with $q$ elements. In this work, we use characters to give the number of rational points on suitable curves of low degree over $\Fq$ in terms of the number of rational points on elliptic curves. In the case where $q$ is a prime number, we give a way to calculate these numbers. As a consequence of these results, we characterize maximal and minimal curves given by equations of the forms $ax^3+by^3+cz^3=0$ and $ax^4+by^4+cz^4=0$.
	\end{abstract}
	
	\maketitle
	
	\section{Introduction} \label{sec1}
	Let $\Fq$ be a finite field with $q=p^k$ elements. For a curve $\Ca$ over $\Fq$, we denote by $N_n(\Ca)$ the number of rational points of $\Ca$ over $\Fqn$. For an irreducible non-singular curve $\Ca$ over $\Fq$, the well-known Riemann Hypothesis \cite[Theorem $3.3$]{moreno1993algebraic} states that the number of rational points on a curve $\Ca$ over $\Fqn$ satisfies
	$$N_n(\mathcal{C})=q^n+1-\sum\limits_{i=1}^{2g} \omega_i^n,$$
	where $g$ denotes the genus of $\Ca$ and $|\omega_i|=\sqrt{q}$ for all $i$. 
	Also by this result, we have the well-known Hasse-Weil bound for the number of rational points on an irreducible non-singular curve over $\Fq$ of genus $g$, given by
	\begin{equation}\label{item42}
	|N_n(\mathcal{C})-q^n-1|\leq 2g\sqrt{q^n}.
	\end{equation}
	In general, calculating the exact value $N_n(\Ca)$ is a difficult task. Many authors have studied curves whose number of rational points attains the upper Hasse-Weil bound, called maximal curves. The number of points on some special curves was studied in \cite{cossidente2000plane,garcia2002curves,hirschfeld1998number, hirschfeld2008algebraic, hu2015number, leep1994number, rojas2013number}. From Riemann Hypothesis, it is possible to deduce another bounds for irreducible singular plane curves. In fact, Aubry and Perret \cite{aubry1996weil} generalized this result to irreducible curves.
	
	\begin{theorem}\cite[Corollary $2.4$]{aubry1996weil}\label{item41}
		The number of points on an irreducible curve $\Ca$ over $\Fq$ is given by
		$$N_n(\mathcal{C})=q^n+1-\sum\limits_{i=1}^{2g} \omega_i^n-\sum\limits_{i=1}^{\Delta_X} \beta_i^n,$$
		where $\Delta_X$ is a constant depending on $\Ca$ and $\omega_i,\beta_i$ are complex numbers. Furthermore, $|\omega_i|=\sqrt{q}$, for all $1\leq i\leq 2g$, and $|\beta_i|=1$, for all $1\leq i\leq \Delta_X$. In addition, $\Delta_X\leq \pi-g$, where $\pi$ is the arithmetic genus of $\Ca$.
	\end{theorem}
	For $q$ an odd prime power, elliptic curves over $\Fq$ are curves given by equations of the form
	$$y^2=ax^3+bx^2+cx+d,$$
	with $18abcd-4b^3d+b^2c^2-4ac^3-27a^2d^2\neq 0$, where $a\neq 0,b,c,d$ are elements of $\Fq$. Results involving points on elliptic curves can be found in \cite{lercier1995counting,schoof1995counting,Silverman}. Since the elliptic curve $\Ca:y^2=ax^3+bx^2+cx+d$ has genus $1$, Theorem \ref{item41} states that there is a complex number $\omega_q(a,b,c,d)$ satisfying
	$$N_n(\Ca)=q^n+1-\omega_q(a,b,c,d)^n-\overline{\omega_q(a,b,c,d)}^n,$$
	where $|\omega_q(a,b,c,d)|=\sqrt{q}$.
	   In Section \ref{sec3}, we show a way to calculate $\omega_q(a,b,c,d)$ computationally faster than direct computation. 
	
	In this work, we use sums of characters to give the number of rational points on suitable curves of degree $3,4$ and $6$ over $\Fq$  in terms of the complex numbers $\omega_q(a,b,c,d)$. Although the clear connection between sums of characters and number of rational points on curves, there exists few articles in literature exploring this relation.  The connection between sums of characters and elliptic curves have already been studied by Williams \cite{williams1979evaluation}.  We use techniques similar to those used by Williams in addition to character properties, associating different forms to count the number of points on the same curve in order to get the exact number of points. In this work, we are interested in affine curves with equation $y^i=f(x)$, where $i=2,3$ or $4$ and $f(x)$ has suitable form. For $n$ a positive integer, let $i$ be a divisor of $q^n-1$ and $\chi_i$ denotes a multiplicative character of order $i$ on $\Fqn^*$. To reduce the notation, we leave implicit the dependence on $n$. It is convenient to extend the domain of the definition of $\chi_i$ from $\Fqn^*$ to $\Fqn$ by setting $\chi_i(0)=1$ if $i=1$ and $\chi_i(0)=0$ if $i\geq 2$. Some of the main results of this paper are the following.
	
	\begin{theorem}\label{item23}
	Let $a,b\in\Fq^*$. The number of rational points on the curve $\Ca: y^3=ax^6+b$ over $\Fqn$ satisfies
	$$N_n(\Ca)=q^n+1-\omega_1^n-\overline{\omega_1}^n-\omega_2^n-\overline{\omega_2}^n-\omega_3^n-\overline{\omega_3}^n-\omega_4^n-\overline{\omega_4}^n-\chi_3(a)-\chi_3^2(a),$$
	where $\omega_1:=\omega_q(a^{-1},0,0,-ba^{-1}),\ \omega_2:=\omega_q(b^{-1},0,0,-ab^{-1}),\ \omega_3:=\omega_q(1,0,0,-4ab)$ and $\omega_4:=\omega_q(-4ab,0,0,1)$.
	\end{theorem}
	
	Despite the following Theorem does not express the number of rational points on the curves in terms of $\omega_q(a,b,c,d)$, this general result will be useful to calculate the number of points on special curves of low degree.

	\begin{theorem}\label{item36}
		Let $A,B,C,a,b,c$ be elements in $\Fq$ and $\alpha_1,\alpha_2\in\F_{q^2}$ be the roots of the polynomial $f(x)=ax^2+bx+c$. Let $\Ca_1: y^i=(Ax^2+Bx+C)(ax^2+bx+c)^{i-1}$ and $\Ca_2: y^i(ax^2+bx+c)=Ax^2+Bx+C$ be two curves over $\overline{\Fq}$. The number of rational points on the curves $\Ca_1$ and $\Ca_2$ over $\Fqn$ satisfies
		$$N_n(\Ca_1)=N_n(\Ca)-\sum\limits_{j=1}^{i-1}\chi_i^j\left(\tfrac{A}{a}\right)-\delta+\gamma$$
		and
		$$N_n(\Ca_2)=N_n(\Ca)+1-\sum\limits_{j=1}^{i-1}\chi_i^j\left(\tfrac{A}{a}\right)-\delta,$$
		where $\Ca$ is the curve over $\overline{\Fq}$ given by the equation $z^2=(B-by^i)^2-4(A-ay^i)(C-cy^i)$, $\gamma:=\um_{\{\alpha_1\in\Fqn\}}+\um_{\{\alpha_1\in\Fqn,\alpha_1\neq\alpha_2\}}$ and $\delta$ is given by
		$$\delta:=\begin{cases}
		1+\chi_2(b^2-4ac),&\text{ if }i=1;\\
		1+\chi_2(4Ac+4Ca-2Bb),&\text{ if }i=2\text{ and }b^2-4ac=0;\\
		1,&\text{ otherwise}.\\
		\end{cases}$$
	\end{theorem}

	\begin{theorem}\label{item19}
		For $a,b,c\in\Fq$, with $a\neq 0$, let $\Ca: y^4=ax^4+bx^2+c$ be a curve over $\Fq$, with $q\equiv 1\pmod{4}$. The number of rational points on $\Ca$ over $\Fqn$ satisfies
		$$N_n(\Ca)\!=\!\begin{cases}
		q^n+1-\omega_1^n-\overline{\omega_1}^{\, n}-\omega_2^n-\overline{\omega_2}^{\, n}-\omega_3^n-\overline{\omega_3}^{\, n},&\text{if }b^2-4ac\neq 0\text{ and }c\neq0;\\
		q^n+1-\omega_1^n-\overline{\omega_1}^{\, n}-\chi_2(b),&\text{if }b^2-4ac\neq 0\text{ and }c=0;\\
		q^n+1-\chi_2(-b/2)+q^n\cdot\chi_2(a),&\text{if }b^2-4ac=0\text{ and }c\neq0;\\
		q^n+1-\chi_2(b)+q^n\cdot\chi_2(a),&\text{if }b^2-4ac=0\text{ and }c=0,\\
		\end{cases}$$
		where $\omega_1:=\omega_q(a^{-1},0,d_1,0)$, $d_1:=\tfrac{b^2-4ac}{4a^2}$, $\omega_2:=\omega_q(c^{-1},0,d_2,0)$, $d_2:=\tfrac{b^2-4ac}{4c^2}$ and $ \omega_3:=\omega_q(a,b,c,0)$.
	\end{theorem}

	In order to prove these main results, we calculate the number of points on many others curves. For all the curves, we give explicitly the complex numbers from Theorem \ref{item41}, as we can see in Theorems \ref{item23} and \ref{item19}.
	
	As a consequence of the results presented throughout the paper, in Section \ref{sec7} we discuss about the maximality and minimality of curves with equations of the form $$ax^n+by^n+cz^n=0$$
	 in the cases $n=3$ and $n=4$, where $a,b,c$ are elements in a prime field $\Fp$, generalizing the conditions presented by Garcia and Tafazolian \cite{garcia2008cartier} in these cases. The case where $a=b=c=1$, the well-known Fermat curve, was discussed in \cite{garcia2008cartier}. More generally, the techniques presented here give a new way to state when an irreducible curve attains the upper bound given in Theorem \ref{item41}.
	 
	The reason why we can not generalize this results to curves given by equations of the form $y^i=f(x)$ of any degree is due to the fact that we do not know how to relate any sum of characters with a elliptic curves. For low degree, there is a natural way to find this relation, as we use in the results in this article. 
	\section{ Preliminaries}\label{sec2}
	
	 In this section, we recall some general results involving rational points on curves over finite fields. Throughout the paper, we use the notation bellow.
	
	\textbf{Notation}
	\begin{itemize}
		\item $\Fq$ is a finite fields with $q=p^k$ elements.
		\item $\overline{\Fq}$ is the algebraic closure of $\Fq$.
		\item For a curve $\mathcal{C}$ over $\Fq$ and $n\geq1$, the set of $\Fqn$-rational points is denoted by $\mathcal{C}(\Fqn)$. 
		\item $N_n(\mathcal{C})$ is the number of rational points of a curve $\mathcal{C}$ over $\Fqn$.
		\item The genus of a curve $\mathcal{C}$ over $\Fq$ is denoted by $g$.
		\item For $D$ a divisor of a curve $\mathcal{C}$, $N(D)=q^{\deg(D)}$ is the \textit{norm} of $D$.
		\item $\overline{\omega}$ denotes the complex conjugate of a complex number $\omega$. 
		\item The indicator function of an event $A$ is denoted by $\um_A$. 
		\item For $n$ a positive integer, let $i$ be a divisor of $q^n-1$ and $\chi_i$ denotes a multiplicative character of order $i$ on $\Fqn^*$
		
	\end{itemize}

	\begin{remark}\label{item8}
	The algebraic curves theory is developed in the projective space, then points at the infinity may be on the curve. For example, the point $(x_0,y_0,z_0)=(0,1,0)$ is on the homogenization of the elliptic curve $y^2=x^3+1$ .Therefore, Theorem \ref{item41} is considering those points at the infinity. In the results that we present in this paper, we follow that convention, considering those rational points.
	\end{remark}

	\begin{definition}
		Let $a\neq 0,b,c,d$ be elements in $\Fq$. The curve $\Ca:y^2=ax^3+bx^2+cx+d$ is called elliptic curve over $\Fq$ if the roots of the polynomial $f(x):=ax^3+bx^2+cx+d$ are distinct.
	\end{definition}
	
	As a direct consequence from Riemann Hypothesis, we have the following result.
	
	\begin{theorem}\label{item9}
		Let $\mathcal{C}:x^2=ay^3+by^2+cy+d$ be an elliptic curve over a finite field $F_q$. There exists a complex number $\omega$, with $|\omega|=\sqrt{q}$, that satisfies
		$$N_n(\mathcal{C})=q^n+1-\omega^n-\overline{\omega}^n.$$
		The number $\omega$ is unique unless conjugation. As $\omega$ depends on $a, b,c,d$ and $q$, we denote $\omega$ by $\omega_q(a,b,c,d)$.
	\end{theorem}
	
	As $\omega_q(a,b,c,d)$ is unique unless conjugation, we let
	\begin{equation}\label{item29}
	\omega_q:\Fq^*\times\Fq^3\longrightarrow \{z\in\C:|z|=\sqrt{q}, \Im(\omega)\geq 0\}
	\end{equation}
	be the function defined by $(a,b,c,d)\mapsto \omega_q(a,b,c,d)$, where $\Im(\omega)$ denotes the imaginary part of $\omega_q$. In Section \ref{sec3}, we present a way to compute $\omega_q(a,b,c,d)$ computationally faster.
	
	\begin{remark}\label{item45}
		By definition of $\omega_q$ and Theorem \ref{item9}, we have $\omega_{q^n}(a,b,c,d)=\omega_q(a,b,c,d)^n$ for all $a,b,c,d\in\Fq$ and for all positive integer $n$.
	\end{remark}

	\section{Rational Points On Elliptic Curves}\label{sec3}
	
	From now, we consider $q$ odd.

	\begin{lemma}\cite[Theorem 5.48]{Lidl}\label{item25} If $a,b,c\in\Fq$, with $a\neq0$ and $\Delta:=b^2-4ac$, then
		$$\sum\limits_{x\in\Fqn} \chi_2(ax^2+bx+c)=\begin{cases}
		-\chi_2(a),&\text{ if }\Delta\neq0;\\
		(q^n-1)\chi_2(a),&\text{ if }\Delta=0.\\
		\end{cases}$$
	\end{lemma}

	\begin{lemma}\cite[Theorem $5.4$]{Lidl}\label{item11} If $\chi$ is a nontrivial multiplicative character of $\Fq^*$, then
	\begin{center}
		$\sum\limits_{g\in\Fq^*}{\chi(g)} = 0.$
	\end{center}
	\end{lemma}

	\begin{remark}
		The discriminant of a cubic polynomial $f(x)=ax^3+bx^2+cx+d$ is given by 
		\begin{equation}\label{item35}
		\Delta=18abcd-4b^3d+b^2c^2-4ac^3-27a^2d^2,
		\end{equation}
		where $\Delta\neq 0$ if and only if the roots of $f(x)$ are distinct. In addition, every root of $f(x)$ is an element of $\Fq$ if $\Delta=0$.
	\end{remark}

	From here, let $\Delta(a,b,c,d)$ denote the discriminant of the polynomial $f(x)=ax^3+bx^2+cx+d$.
	
	\begin{lemma}\label{item10} Let $a,b,c,d\in\Fq$ be elements in $\Fq$, with $a\neq0$. If $\alpha_1,\alpha_2,\alpha_3\in\F_{q^6}$ are the roots of the polynomial $f(x):=ax^3+bx^2+cx+d$, then
		$$\sum\limits_{x\in\Fqn} \chi_2\big(f(x)\big)=\begin{cases}
		-\omega_q(a,b,c,d)^n-\overline{\omega_q(a,b,c,d)}^{\, n},&\text{ if }\Delta\neq0;\\
		-\chi_2(a\alpha_1-a\alpha_2),&\text{ if }f(x)=a(x-\alpha_1)^2(x-\alpha_2);\\
		0,&\text{ if }f(x)=a(x-\alpha_1)^3,\\
		\end{cases}$$
	where $\Delta:=\Delta(a,b,c,d)$ as defined in Equation \eqref{item35} and the complex number $\omega_q(a,b,c,d)$ is defined as in Theorem \ref{item9}.
	\end{lemma}
	\begin{proof}
		Let $\Ca$ be the curve $y^2=ax^3+cx^2+cx+d$. We observe that
		$$N_n(\mathcal{C})=1+\sum\limits_{x\in\Fqn}\left[1+\chi_2(ax^3+bx^2+cx+d)\right].$$
		By Theorem \ref{item9}, if $\Delta\neq 0$, we have
		$$N_n(\mathcal{C})=1+\sum\limits_{x\in\Fqn}\left[1+\chi_2(ax^3+bx^2+cx+d)\right]=q^n+1-\omega_q(a,b,c,d)^n-\overline{\omega_q(a,b,c,d)}^{\, n},$$
		As 
		$\sum_{x\in\Fqn} 1=q^n,$
		the result follows. If $\Delta=0$, then there exists a root $\alpha\in\mathbb{F}_{q}$ of the polynomial $ax^3+bx^2+cx+d$ with multiplicity $\geq 2$. Hence, $ax^3+bx^2+cx+d=a(x-\alpha)^2(x-\beta)$, with $\alpha,\beta\in\mathbb{F}_{q}$. Thus, from the relation
		$$\chi_2(ax^3+bx^2+cx+d)\!=\chi_2\left((x-\alpha)^2\right)\chi_2\big(a(x-\beta)\big)=\begin{cases}
		\chi_2\big(a(x-\beta)\big),&\text{ if }x\neq \alpha;\\
		0,&\text{ if }x= \alpha,\\
		\end{cases}$$
		and Lemma \ref{item11}, we have
		$$\begin{aligned}
		\sum\limits_{x\in\Fqn} \chi_2(ax^3+bx^2+cx+d)&=\sum\limits_{x\in\Fqn\backslash\{\alpha\}}\chi_2\big(a(x-\alpha)^2(x-\beta)\big)\\
		&=\sum\limits_{x\in\Fqn\backslash\{\alpha\}}\chi_2(ax-a\beta)\\
		&=-\chi_2(a\alpha-a\beta).
		\end{aligned}\\$$
		This complete the proof.$\hfill\qed$
		
	\end{proof}

	As we use Lemma \ref{item10} in most results in this paper, it is convenient to introduce the following notation. We define the function $\Delta':\Fq^4\rightarrow\{1,2,3\}$ by letting
	$$\Delta'(a,b,c,d)=\begin{cases}
	1,&\text{ if }\alpha_i\neq\alpha_j\text{ for all }1\leq i<j\leq 3;\\
	2,&\text{ if }\alpha_i=\alpha_j\neq\alpha_k\text{ for some }\{i,j,k\}=\{1,2,3\};\\
	3,&\text{ if }\alpha_1=\alpha_2=\alpha_3,
	\end{cases}$$
	where $\alpha_1,\alpha_2,\alpha_3$ are the roots of the polynomial $f(x):=ax^3+bx^2+cx+d$. From this definition, it follows that $\Delta(a,b,c,d)\neq0$ if and only if $\Delta'(a,b,c,d)=1$. We also define the function $\alpha:\{(x_1,x_2,x_3,x_4)\in\Fq^4:\Delta'(a,b,c,d)=2\}\rightarrow\Fq$ by 
	$$\alpha(a,b,c,d)= a\alpha_1-a\alpha_2,$$
	where $f(x)=a(x-\alpha_1)^2(x-\alpha_2)$.

	\begin{lemma}\cite[Lemma $7.3$]{Lidl}\label{item2}
		Let $m$ be a positive integer. We have
		$$\sum_{a\in\Fq}a^m=\begin{cases}
		0,\text{ if }(q-1)\nmid m\text{ or }m=0;\\
		-1,\text{ if }(q-1)\mid m\text{ and }m\neq0,\\
		\end{cases}$$
		where $0^0:=1$.
	\end{lemma}
	\begin{remark}\label{item1}
	It is known that $\chi_2(a)=(-1)^n\in\C$ if and only if $a^{\frac{q-1}{2}}=(-1)^n\in\Fq$.
	Indeed, by definition, $\chi_2(b)=-1$ if and only if $b$ is not a square in $\F_q^*$, so $b=\theta^i$, where $\theta$ is a generator of the group $\F_q^*$ and $i$ is odd. It follows that $b^{\frac{q-1}2}=(\theta^{\frac{q-1}2})^i=(-1)^i=-1$.  In similar way, we have that $\chi_2(b)=-1$ if and only if $b^{\frac{q-1}2}=1$. 
	\end{remark}
	
	The following lemma characterizes, modulo $p$, the trace of Frobenius
	$$\tau_q(a,b,c,d):=\omega_q(a,b,c,d)+\overline{\omega_q(a,b,c,d)}$$
	of an elliptic curve given by equation $y^2=ax^3+bx^2+cx+d$ over $\Fp$. As we are interested in odd characteristic, we suppose $b=0$. This lemma enables us to calculate the number of points on elliptic curves whose coefficients are in a prime field $\Fp$.

	\begin{lemma}\label{item12}
		Let $\Ca: y^2=Ax^3+Bx+C$ be  an elliptic curve over $\Fpn$, where $p$ is an odd prime. The trace of Frobenius of $\Ca$ satisfies
the relation
		$$\tau_{p^n}(A,0,B,C)\equiv \sum_{l=\lceil\frac{p^n-1}{6}\rceil}^{\lfloor \frac{p^n-1}{4}\rfloor} \binom{\frac{p^n-1}{2}}{2l}\binom{2l}{\frac{p^n-1-2l}{2}}A^{\frac{p^n-1}{2}-l}B^{3l-\frac{p^n-1}{2}}C^{\frac{p^n-1}{2}-2l}\pmod{p}.$$
		
	\end{lemma}
	
	\begin{proof}
		We observe that
		$$1+\chi(Ax^3+Bx+C)=\begin{cases}
		0,\text{ if }Ax^3+Bx+C\text{ is not a square in }\Fpn;\\
		2,\text{ if }Ax^3+Bx+C\text{ is a square in }\Fpn.\\
		\end{cases}$$ 
		Therefore, we have
		$$N_1(\Ca)=1+\sum_{x\in\Fpn}[1+\chi(Ax^3+Bx+C)]=p^n+1+\sum_{x\in\Fpn}\chi(Ax^3+Bx+C).$$
		By  Remark \ref{item1},
		$$\begin{aligned}\sum_{x\in\Fpn}\chi(x^3+Ax+B)
		&\equiv \sum_{x\in\Fpn}(Ax^3+Bx+C)^{\frac{p^n-1}{2}}\\
		&\equiv \sum_{x\in\Fpn}\sum_{i=0}^{\frac{p^n-1}{2}}\binom{\frac{p^n-1}{2}}{i}(Ax^3+Bx)^i C^{\frac{p^n-1}{2}-i}\\
		&\equiv \sum_{x\in\Fpn}\sum_{i=0}^{\frac{p^n-1}{2}}\binom{\frac{p^n-1}{2}}{i}C^{\frac{p^n-1}{2}-i}\sum_{j=0}^{i}\binom{i}{j}A^j B^{i-j}x^{2j+i}\\
		&\equiv \sum_{i=0}^{\frac{p^n-1}{2}}\sum_{j=0}^{i}\binom{\frac{p^n-1}{2}}{i}\binom{i}{j}A^jB^{i-j}C^{\frac{p^n-1}{2}-i}\sum_{x\in\Fp}x^{2j+i}\pmod{p}.\\
		\end{aligned}$$
		
		By Lemma \ref{item2}, the  sum  $\sum_{x\in\Fpn}x^{2j+i}$ is nonzero only if $2j+i\equiv 0\pmod{p^n-1}$ and $2j+i\neq 0$, and in these cases, the sum is $-1$. In addition, since  
		$$2j+i\le 2\cdot\frac{p^n-1}{2}+\frac{p^n-1}{2}=\frac{3(p^n-1)}{2}< 2(p^n-1),$$
it follows that 
		$$\begin{aligned}N_1(\Ca)-p^n-1
		&\equiv\sum_{i=0}^{\frac{p^n-1}{2}}\sum_{j=0}^{i}\binom{\frac{p^n-1}{2}}{i}\binom{i}{j}A^jB^{i-j}C^{\frac{p^n-1}{2}-i}\sum_{x\in\Fpn}x^{2j+i}\\
		&\equiv -\sum_{l=\lceil\frac{p^n-1}{6}\rceil}^{\lfloor \frac{p^n-1}{4}\rfloor} \binom{\frac{p^n-1}{2}}{2l}\binom{2l}{\frac{p^n-1-2l}{2}}A^{\frac{p^n-1-2l}{2}}B^{2l-\frac{p^n-1-2l}{2}}C^{\frac{p^n-1}{2}-2l}\\
		&\equiv -\sum_{l=\lceil\frac{p^n-1}{6}\rceil}^{\lfloor \frac{p^n-1}{4}\rfloor} \binom{\frac{p^n-1}{2}}{2l}\binom{2l}{\frac{p^n-1-2l}{2}}A^{\frac{p^n-1-2l}{2}}B^{3l-\frac{p^n-1}{2}}C^{\frac{p^n-1}{2}-2l}\pmod{p}.\\
		\end{aligned}$$ 
		Since $N_1(\Ca)=p^n+1-\tau_{p^n}(A,0,B,C)$, the result follows.
		$\hfill\qed$
		
	\end{proof}

	In Lemma \ref{item12} there is an abuse of language when we write 
	$$\sum_{x\in\Fp}\chi(x^3+Ax+B)\equiv \sum_{x\in\Fp}(Ax^3+Bx+C)^{\frac{p-1}{2}}$$
	since the left summation is over $\C$ and the right summation is over $\Fq$. In fact, we do it many times throughout this paper.

	Since $|\tau(A,0,B,C)|=|\omega_p(A,0,B,C)+\overline{\omega_p(A,0,B,C)}|\leq\lfloor2\sqrt{p}\rfloor$ (by Theorem \ref{item9}), we can use Lemma \ref{item12} to calculate the complex number $\omega_p(A,0,B,C)$ in the case where $p\geq 17$, as in the following example.
	
\begin{example}\label{item39}
	Let $\mathcal{J}:x^2=y^3+2$ be an elliptic curve over $\F_{19}$. From Lemma \ref{item12},
	$$\omega_{19}(1,0,0,2)+\overline{\omega_{19}(1,0,0,2)}\equiv\sum\limits_{l=3}^{4}\binom{9}{2l}\binom{2l}{9-l}0^{3l-9}2^{9-2l}\equiv 7\pmod{19}.$$
	Since $|\omega_{19}(1,0,0,2)+\overline{\omega_{19}(1,0,0,2)}|\leq\lfloor2\sqrt{19}\rfloor=8$, we have $\omega_{19}(1,0,0,2)+\overline{\omega_{19}(1,0,0,2)}=7$. Thus, using that $|\omega_{19}(1,0,0,2)|=\sqrt{19}$, we must have
	$$\omega_{19}(1,0,0,2)=\frac{7}{2}+i \sqrt{19-\frac{7^2}{2^2}}=\frac{7}{2}+i \sqrt{\frac{27}{4}}.$$
	In addition, the number of rational points of $\Ja$ over $\F_{19^n}$ is given by
	$$N_n(\Ja)=19^n+1-\left(\frac{7}{2}+i \sqrt{\frac{27}{4}}\right)^n-\left(\frac{7}{2}-i \sqrt{\frac{27}{4}}\right)^n.$$
\end{example}

We will use this technique in the examples of this paper in order to calculate the number of rational points on suitable curves. In \cite{sun2013congruences}, the author present congruences similar to congruence in Lemma \ref{item12}. In fact, some values of $\omega_q$ are well-known in the case where $p$ is a prime number, e.g. see Theorem $6.2.9$ and Theorem $6.2.10$ in \cite{berndt1998gauss}.

	\section{Rational Points on Curves of the Form $y^2=f(x)$}\label{sec4}
	
		Throughout this section, for an event $A$, let 
		$$\um_A:=\begin{cases}
		1,&\text{ if }A\text{ occurs};\\
		0,&\text{ if }A\text{ does not occur}\\
		\end{cases}$$
		be the indicator function of the event $A$. In algebraic geometry, a hyperelliptic curve of genus $g\geq 1$ is an algebraic curve given by equation
		$$ y^{2}+h(x)y=f(x),$$
		where $f(x)$ is a polynomial of degree $2g+1$ or $2g+2$ with distinct roots and $h(x)$ is a polynomial of degree $\leq g+1$. When the characteristic of the field is not $2$, we can take $h(x)=0$ (make the change of variables, by taking $y=z-\frac{1}{2}h(x)$ and $x=w$). Hyperelliptic curves are  useful in cryptography (for example, see \cite{cohen2005handbook}). In Ulas \cite{ulas2007rational} and Nelson, Solymosi, Tom and Wong \cite{nelsonnumber}, the authors present results concerning the number of rational points on hyperelliptic curves. In this Section, we present the number of rational points on curves of the form $y^2=f(x)$, where $f(x)\in\Fq$ is a suitable polynomial of degree $4$ or $6$. In particular, we give the number of points on most hyperelliptic curves of genus $1$ when $q$ is odd.

		\begin{remark}\label{item32}
			Let $\Lambda_1,\Lambda_2\subset\Fq$ be two sets with the same number of elements. Let $f$ be a bijective map from $\Lambda_1$ to $\Lambda_2$ and $g:\Lambda_2\rightarrow\Fq$ an arbitrary function. Then 
			$$\sum\limits_{x\in\Lambda_1}\chi_i\big(g(f(x))\big)=\sum\limits_{z\in\Lambda_2}\chi_i\big(g(z)\big).$$
		\end{remark}
	
		\begin{lemma}\label{item15}
			 If $a,b,c,d\in\Fq$, with $a\neq0$ and $d\neq0$, then
		$$\sum\limits_{x\in\Fqn} \chi_2(ax^3+bx^2+cx+d)\chi_2(x)=\begin{cases}
		-\omega_q(d,c,b,a)^n-\overline{\omega_q(d,c,b,a)}^{\, n}-\chi_2(a),&\text{ if }\Delta'=1;\\
		-\chi_2(\alpha)-\chi_2(a),&\text{ if }\Delta'= 2;\\
		-\chi_2(a),&\text{ if }\Delta'= 3,\\
		\end{cases}$$
		where $\Delta':=\Delta'(a,b,c,d)$, $\alpha:=\alpha(a,b,c,d)$ and $\omega_q(d,c,b,a)$ is defined as in Theorem \ref{item9}.
	\end{lemma}

	\begin{proof}
		Using the fact that $\chi_2(0)=0$, we have 
		$$\begin{aligned}
		\sum\limits_{x\in\Fqn} \chi_2(ax^3+bx^2+cx+d)\chi_2(x)
		& =\sum\limits_{x\in\Fqn^*} \chi_2(ax^3+bx^2+cx+d)\chi_2(x^{-3})\\
		& =\sum\limits_{x\in\Fqn^*} \chi_2(a+bx^{-1}+cx^{-2}+dx^{-3})\\
		& =\sum\limits_{z\in\Fqn^*} \chi_2(a+bz+cz^2+dz^3).\\
		\end{aligned}$$
		By Lemma \ref{item10}, 
		$$\sum\limits_{z\in\Fqn^*} \chi_2(a+bz+cz^2+dz^3)=\begin{cases}
		-\omega_q(d,c,b,a)^n-\overline{\omega_q(d,c,b,a)}^{\, n}-\chi_2(a),&\text{ if }\Delta'=1;\\
		-\chi_2(\alpha)-\chi_2(a),&\text{ if }\Delta'= 2;\\
		-\chi_2(a),&\text{ if }\Delta'= 3,\\
		\end{cases}$$
		where $\Delta':=\Delta'(a,b,c,d)$ and $\alpha:=\alpha(a,b,c,d)$.$\hfill\qed$
	\end{proof}

	\begin{theorem}\label{item16}
		For $a,b,c,d,e\in\Fq$, with $a\neq0$, let $\Ca: y^2=(ax^3+bx^2+cx+d)(x+e)$ be a curve over $\overline{\Fq}$. The number of rational points on $\Ca$ over $\Fqn$ satisfies
		$$N_n(\Ca)=\begin{cases}
		q^n+1-\omega_q(d',c',b',a')^n-\overline{\omega_q(d',c',b',a')}^{\, n}-\chi_2(a'),&\mkern-16mu\text{ if }d'\neq0\text{ and }\Delta'=1;\\
		q^n+1-\chi_2(\alpha)-\chi_2(a'),&\mkern-16mu\text{ if }d'\neq0\text{ and }\Delta'= 2,\\
		q^n+1-\chi_2(a'),&\mkern-16mu\text{ if }d'\neq0\text{ and }\Delta'= 3;\\
		q^n+1-\chi_2(a')-\chi_2(c'),&\mkern-16mu\text{ if }d'=0\text{ and }(b')^2\neq4a'c';\\
		q^n+1+(q^n-1)\chi_2(a')-\chi_2(c'),&\mkern-16mu\text{ if }d'=0\text{ and }(b')^2=4a'c';\\
		
		\end{cases}$$
		where $a'=a$, $b'=b-3ae$, $c'=3ae^2-2eb+c, d'=be^2-ae^3+d-ec$, $\Delta':=\Delta'(a',b',c',d')$ and $\alpha:=\alpha(a',b',c',d')$.
	\end{theorem}

	\begin{proof}
		We have 
		$$\begin{aligned}
		N_n(\Ca)
		& =1+\sum\limits_{x\in\Fqn} \left[1+\chi_2(ax^3+bx^2+cx+d)\chi_2(x+e)\right]\\
		&=q^n+1+\sum\limits_{z\in\Fqn}\chi_2(a(z-e)^3+b(z-e)^2+c(z-e)+d)\chi_2(z)\\
		&=\!q^n\!+1\!+\!\!\sum\limits_{z\in\Fqn}\chi_2(a(z^3\!-3z^2e+3ze^2\!-e^3)\!+b(z^2\!-2ze+e^2)+c(z\!-e)\!+d)\chi_2(z)\\
		&=\!q^n\!+\!1\!+\!\!\sum\limits_{z\in\Fqn}\!\chi_2(az^3\!+(b\!-3ae)z^2+(3ae^2\!-2eb+c)z+be^2\!-ae^3+d\!-ec)\chi_2(z).\\
		\end{aligned}$$
		Hence, the result follows by Lemma \ref{item15} in the case $be^2-ae^3+d-ec\neq0$. Otherwise, 
		$$\begin{aligned}
		N_n(\Ca)
		&=q^n+1+\sum\limits_{z\in\Fqn^*}\chi_2(az^3+(b-3ae)z^2+(3ae^2-2eb+c)z)\chi_2(\tfrac{1}{z})\\
		&=q^n+1+\sum\limits_{z\in\Fqn^*}\chi_2(az^2+(b-3ae)z+(3ae^2-2eb+c)).\\
		\end{aligned}$$
		In this case, the result follows by Lemma \ref{item25}.
		$\hfill\qed$
	\end{proof}

	\begin{example}
		Let $\Ja: y^2=(x^3+x^2-x+1)(x+1)$ be a curve over $\F_{73}$. In order to calculate $\omega_{73}(2,0,-2,1)$, we use Lemma \ref{item12} as in Example \ref{item39}. We note that
		$$\omega_{73}(2,0,-2,1)+\overline{\omega_{73}(2,0,-2,1)}\equiv \sum_{l=12}^{18} \binom{36}{2l}\binom{2l}{36-l}2^{36-l}(-2)^{3l-36}1^{36-2l}\equiv 16\pmod{73},$$
		then
		$\omega_{73}(2,0,-2,1)=8+3i $. Since $\Delta'(2,0,-2,1)=1$, Theorem \ref{item16} states that the number of rational points on $\Ja$ over $\F_{73^n}$ is given by
		$$N_n(\Ja)=73^n-\left(8+3i\right)^n-\left(8+3i\right)^n.$$
	\end{example}

	\begin{theorem}\label{item17}
		For $a,b,c,d\in\Fq$, with $a\neq 0$, let $\Ca: y^2=ax^6+bx^4+cx^2+d$ be a curve over $\Fq$. The number of rational points on $\Ca$ over $\Fqn$ is given by
		$$N_n(\Ca)=\begin{cases}
		q^n+1-\omega_1^n-\overline{\omega_1}^{\, n}-\omega_2^n-\overline{\omega_2}^{\, n}-\chi_2(a),&\text{ if }d\neq0\text{ and }\Delta'=1;\\
		q^n+1-\chi(\alpha_1)-\chi(\alpha_2)-\chi_2(a),&\text{ if }d\neq0\text{ and }\Delta'= 2;\\
		q^n+1-\chi_2(a),&\text{ if }d\neq0\text{ and }\Delta'= 3;\\
		q^n+1-\omega_1^n-\overline{\omega_1}^{\, n}-\chi_2(a)-\chi_2(c),&\text{ if }d=0\neq c\text{ and }b^2-4ac\neq0;\\
		q^n+1-\chi(\alpha_1)+(q^n-1)\chi(a)-\chi_2(c),&\text{ if }d=0\text{ and }b^2-4ac=0;\\
		q^n+1-\chi(\alpha_1)-\chi_2(a)-\chi_2(c),&\text{ if }c=d=0\text{ and }b^2-4ac\neq0,\\
		\end{cases}$$
		where $\omega_1:=\omega_q(a,b,c,d)$, $\omega_2:=\omega_q(d,c,b,a),\alpha_1:=\alpha(a,b,c,d),\alpha_2:=\alpha(d,c,b,a)$ and $\Delta':=\Delta'(a,b,c,d)$.
	\end{theorem}

	\begin{proof}
		The number of rational points on $\Ca$ is given by
		$$\begin{aligned}
		N_n(\Ca)&=1+\sum\limits_{x\in\Fqn}\left[1+\chi_2(ax^3+bx^2+cx+d)\right]\big[1+\chi_2(x)\big]\\
		&=q^n+1+\sum\limits_{x\in\Fqn}\left[\chi_2(ax^3+bx^2+cx+d)+\chi_2(x)+\chi_2(ax^3+bx^2+cx+d)\chi_2(x)\right].\\
		\end{aligned}$$
		The result follows by Lemma \ref{item10}, Lemma \ref{item11} and Lemma \ref{item15} in the case $d\neq0$. Otherwise,
		$$\begin{aligned}
		N_n(\Ca)&=q^n+1+\sum\limits_{x\in\Fqn}\chi_2(ax^3+bx^2+cx)+\sum\limits_{x\in\Fqn^*}\chi_2(ax^3+bx^2+cx)\chi_2(\tfrac{1}{x})\\
		&=q^n+1+\sum\limits_{x\in\Fqn}\chi_2(ax^3+bx^2+cx)+\sum\limits_{x\in\Fqn^*}\chi_2(ax^2+bx+c).\\
		\end{aligned}$$
		Hence, the result follows by Lemma \ref{item25} and Lemma \ref{item10}.
		$\hfill\qed$
	\end{proof}

	\begin{example}
		Let $\Ja: y^2=2x^6+1$ be a curve over $\F_{29}$. In order to calculate $\omega_{29}(2,0,0,1)$ and $\omega_{29}(1,0,0,2)$, we use Lemma \ref{item12} as in Example \ref{item39}. We note that
		$$\omega_{29}(2,0,0,1)+\overline{\omega_{29}(2,0,0,1)}\equiv \sum_{l=5}^{7} \binom{14}{2l}\binom{2l}{14-l}2^{14-l}0^{3l-14}1^{14-2l}\equiv 0\pmod{29}$$
		and
		$$\omega_{29}(1,0,0,2)+\overline{\omega_{29}(1,0,0,2)}\equiv \sum_{l=5}^{7} \binom{14}{2l}\binom{2l}{14-l}1^{14-l}0^{3l-14}2^{14-2l}\equiv 0\pmod{29},$$
		then
		$\omega_{29}(2,0,0,1)=\omega_{29}(1,0,0,2)=i\sqrt{29} $. Since $2$ is not square residue in $\F_{29}$ and $\Delta'(2,0,0,1)=1$, Theorem \ref{item17} states that the number of rational points on $\Ja$ over $\F_{29^n}$ is given by
		$$N_n(\Ja)=\begin{cases}
		29^n+4\cdot 29^{2k+1},&\text{ if }n=4k+2\text{ for an integer }k;\\
		29^n-4\cdot 29^{2k},&\text{ if }n=4k\text{ for an integer }k;\\
		29^n+2,&\text{ if }n\equiv 1\pmod{2}.\\
		\end{cases}$$
	\end{example}

	\begin{theorem}\label{item30}
		For $a,b,c\in\Fq$, with $a\neq0$, let $\Ca: y^2=ax^4+bx^2+c$ be a curve over $\overline{\Fq}$. The number of rational points on $\Ca$ over $\Fqn$ is given by
		$$N_n(\Ca)=\begin{cases}
		q^n+1-\omega_q(a,b,c,0)^n-\overline{\omega_q(a,b,c,0)}^{\, n}-\chi_2(a),&\text{ if }b^2-4ac\neq 0\text{, with }c\neq0;\\
		q^n+1-\chi_2(b)-\chi_2(a),&\text{ if }b^2-4ac\neq 0\text{, with }c=0;\\
		q^n+1-\chi_2(-b/2)+(q^n-1)\chi_2(a),&\text{ if }b^2-4ac= 0\text{, with }c\neq0;\\
		q^n+1+(q^n-1)\chi_2(a),&\text{ if }b^2-4ac= 0\text{, with }c=0.\\
		\end{cases}$$
		
	\end{theorem}

	\begin{proof}
		As in the last theorem, 
		$$\begin{aligned}
		N_n(\Ca)&=q^n+1+\sum\limits_{x\in\Fqn}\left[\chi_2(ax^2+bx+c)+\chi_2(x)+\chi_2(ax^3+bx^2+cx)\right].\\
		\end{aligned}$$
		By Lemma \ref{item25}, Lemma \ref{item11} and Lemma \ref{item10}, the result follows.$\hfill\qed$
	\end{proof}
	
	\begin{lemma}\label{item34} Let $i$ be a divisor of $q^n-1$ and $A,B,C,a,b,c\in\Fq$, with $A\neq0$ and $a\neq0$. If $\alpha_1,\alpha_2\in\F_{q^2}$ are the roots of the polynomial $f(x)=ax^2+bx+c$ and the polynomials $f(x)$ and $g(x)=Ax^2+Bx+C$ has no common roots, then
		$$\sum\limits_{\substack{x\in\Fqn\\ x\not\in\{\alpha_1,\alpha_2\}}}\mkern-15mu\left[1\!+\!\chi_i\left(\tfrac{Ax^2+Bx+C}{ax^2+bx+c}\right)\!+\dots+\chi_i^{i-1}\left(\tfrac{Ax^2+Bx+C}{ax^2+bx+c}\right)\right]=
		N_n(\Ca)-\delta-\left[1+\dots+\chi_i^{i-1}\left(\tfrac{A}{a}\right)\right],$$
		where $\Ca$ is a curve given by the equation $z^2=(B-by^i)^2-4(A-ay^i)(C-cy^i)$ and $\delta$ is a constant referent to points at the infinity given by
		$$\delta:=\begin{cases}
		1+\chi_2(b^2-4ac),&\text{ if }i=1;\\
		1+\chi_2(4Ac+4Ca-2Bb),&\text{ if }i=2\text{ and }b^2-4ac=0;\\
		1,&\text{ otherwise}.\\
		\end{cases}$$
	\end{lemma}
	
	\begin{proof} We observe that the summation
		$$\sum\limits_{\substack{x\in\Fqn\\ x\not\in\{\alpha_1,\alpha_2\}}}\mkern-15mu\left[1\!+\!\chi_i\left(\tfrac{Ax^2+Bx+C}{ax^2+bx+c}\right)\!+\dots+\chi_i^{i-1}\left(\tfrac{Ax^2+Bx+C}{ax^2+bx+c}\right)\right]$$
		 count the number of rational points on the curve $\Ca':y^i(ax^2+bx+c)=Ax^2+Bx+C$ over $\Fqn$. Let $y_0^i\in\Fqn$, if there exists $x_0\in\Fqn$ such that $(x_0,y_0)$ is on the curve, then it is given by one of the following
		$$\frac{by^i-B\pm\sqrt{(B-by^i)^2-4(A-ay^i)(C-cy^i)}}{2(A-ay^i)}.$$
		Hence, fixing $y_0^i\in\Fqn$, there exists $x_0\in\Fqn$ such that $(x_0,y_0)$ is on the curve $\Ca'$ only if $(B-by_0^i)^2-4(A-ay_0^i)(C-cy_0^i)$ is a square in $\Fqn$ or if $y_0^i=Aa^{-1}$. Let $\Ca$ be the curve given by equation $z^2=(B-by^i)^2-4(A-ay^i)(C-cy^i)$. In the case where $Aa^{-1}$ is a $i$-th power in $\Fqn$ and $Aa^{-1}b\neq B$, there are $i$ rational points of the form $\left(\tfrac{cA-Ca}{Ba-bA},\gamma^i\right)$ on the curve $\Ca'$, where $\gamma$ is a $i$-th primitive root of $Aa^{-1}$.  Since there are $2i$ rational points in $\Ca$ where $y^i=Aa^{-1}$ (namely $\left(\pm\tfrac{Ba-bA}{a} ,\gamma^i\right)$), the result is proved in this case. In the case where $Aa^{-1}b= B$, there are no rational points with $y^i=Aa^{-1}$ on the curve $\Ca'$ and only $i$ on $\Ca$, then, in the same way, the result follows. $\hfill\qed$
	
	\end{proof}

	\begin{remark}\label{item37}
		Let $b,c,A,B,C\in\Fq$, with $A\neq 0$, $b\neq0$ and $A(\tfrac{-c}{b})^2+B(\tfrac{-c}{b})+C\neq0$. For $i$ a divisor of $q^n-1$, in the same way of the proof of Lemma \ref{item34}, we have
		$$\sum\limits_{\substack{x\in\Fqn\\ x\neq -c/b}}\mkern-15mu\left[1\!+\!\chi_i\left(\tfrac{Ax^2+Bx+C}{bx+c}\right)\!+\dots+\chi_i^{i-1}\left(\tfrac{Ax^2+Bx+C}{bx+c}\right)\right]=\begin{cases}
		N_n(\Ca)-2,&\text{ if }i=1;\\
		N_n(\Ca)-1,&\text{ otherwise. }\\
		\end{cases}
		$$
		where $\Ca$ is a curve given by the equation $z^2=(B-by^i)^2-4A(C-cy^i).$
	\end{remark}

	\begin{proof}(proof of \ref{item36})
		By Lemma \ref{item34}, 
		$$\begin{aligned}
		N_n(\Ca_1)&=1+\sum\limits_{x\in\Fqn}\left[1+\cdots+\chi_i^{i-1}\big((Ax^2+Bx+C)(ax^2+bx+c)^{i-1}\big)\right]\\
		&=1+\um_{\{\alpha_1\in\Fqn\}}+\um_{\{\alpha_1\in\Fqn,\alpha_1\neq\alpha_2\}}+\mkern-8mu\sum\limits_{x\in\Fqn\backslash\{\alpha_1,\alpha_2\}}\!\left[1+\cdots+\chi_i^{i-1}\!\!\left(\frac{Ax^2+Bx+C}{ax^2+bx+c}\right)\right]\\
		&=1+\um_{\{\alpha_1\in\Fqn\}}+\um_{\{\alpha_1\in\Fqn,\alpha_1\neq\alpha_2\}}+N_n(\Ca)-\left[1+\chi_i\left(\tfrac{A}{a}\right)+\cdots+\chi_i^{i-1}\left(\tfrac{A}{a}\right)\right]-\delta,\\
		\end{aligned}$$
		where $\Ca$ is the curve given by the equation $z^2=(B-by^i)^2-4(A-ay^i)(C-cy^i)$. In the same way,
		$$\begin{aligned}N_n(\Ca_2)&=2+\sum\limits_{x\in\Fqn\backslash\{\alpha_1,\alpha_2\}}\left[1+\chi_i\left(\frac{Ax^2+Bx+C}{ax^2+bx+c}\right)+\cdots+\chi_i^{i-1}\left(\frac{Ax^2+Bx+C}{ax^2+bx+c}\right)\right]\\
		&=2+N_n(\Ca)-\left[1+\chi_i\left(\tfrac{A}{a}\right)+\cdots+\chi_i^{i-1}\left(\tfrac{A}{a}\right)\right]-\delta.\\
		\end{aligned}$$
		$\hfill\qed$
	\end{proof}

	We have the necessary tools to calculate $N_n(\Ca)$ in some cases, as we will see in the following results.

	\begin{corollary}\label{item31}
		Let $A,B,C,a,b,c$ be elements in $\Fq$ that satisfy the hypothesis of the Lemma \ref{item34}. Assuming $b^2-4ac\neq0$ and $B^2-4AC\neq 0$, the number of rational points on the curve $\Ca: y^2(ax^2+bx+c)=Ax^2+Bx+C$ over $\Fqn$ is given by
		$$N_n(\Ca)=\begin{cases}
		q^n+1-\omega_q(a',b',c',0)^n-\overline{\omega_q(a',b',c',0)}^{\, n}-\chi_2(a')-\chi_2\left(\tfrac{A}{a}\right),&\mkern-10mu\text{ if }\Delta\neq 0;\\
		q^n+1-\chi_2(\alpha)+(q^n-1)\chi_2(a')-\chi_2\left(\tfrac{A}{a}\right),&\mkern-10mu\text{ if }\Delta= 0,\\
		\end{cases}$$
		where $a':=b^2-4ac,\ b':=4Ac+4Ca-2Bb,\ c':=B^2-4AC$, $\Delta:=\Delta(a',b',c',0)$ and $\alpha=\alpha(a',b',c',0)$.
	\end{corollary}

	\begin{proof} It follows from Theorem \ref{item36} and Theorem \ref{item30}.
		$\hfill\qed$
	\end{proof}
	
	\begin{corollary}\label{item18}
		Let $A,B,C,a,b,c$ be elements in $\Fq$ that satisfy the hypothesis of the Lemma \ref{item34}. Assuming $b^2-4ac\neq0$ and $B^2-4AC\neq 0$, the number of rational points on the curve $\Ca: y^2=(ax^2+bx+c)(Ax^2+Bx+C)$ over $\Fqn$ is given by
		$$N_n(\Ca)\!=\!\begin{cases}
		q^n\!+2\cdot\um_{\{\alpha_1\in\Fqn\}}\!-\omega_q(a',b',c',0)^n\!-\overline{\omega_q(a',b',c',0)}^{\, n}\!\!-\chi_2(a')-\chi_2\left(\tfrac{A}{a}\right),&\mkern-17mu\text{ if }\Delta\neq 0;\\
		q^n\!+2\cdot\um_{\{\alpha_1\in\Fqn\}}-\chi_2(\alpha)+(q^n-1)\chi_2(a'),&\mkern-17mu\text{ if }\Delta= 0,\\
		\end{cases}$$
		where $a':=b^2-4ac,\ b':=4Ac+4Ca-2Bb,\ c':=B^2-4AC$, $\Delta:=\Delta(a',b',c',0)$ and $\alpha=\alpha(a',b',c',0)$.
	\end{corollary}

\begin{proof} It follows from Theorem \ref{item36} and Theorem \ref{item30}.
$\hfill\qed$
\end{proof}

The result in Corollary \ref{item18} generalize the sums of quadratic characters studied by Williams in \cite{williams1979evaluation}, where the author calculate the sum
$$\sum\limits_{x\in\Fp} \chi_2\left((ax^2+bx+c)(Ax^2+Bx+C)\right)$$
	over a prime field $\Fp$.
	\begin{example}
		Let $\Ja: y^2=(x^2+3x+2)(x^2-2x-5)$ be a curve over $\F_{67}$. In order to calculate $\omega_{67}(1,0,24,0)$, we use Lemma \ref{item12} as in Example \ref{item39}. We note that
		$$\omega_{67}(1,0,24,0)+\overline{\omega_{67}(1,0,24,0)}\equiv \sum_{l=11}^{16} \binom{33}{2l}\binom{2l}{33-l}1^{33-l}(24)^{3l-33}0^{33-2l}\equiv 0\pmod{67},$$
		then
		$\omega_{67}(1,0,24,0)=i\sqrt{67} $. Since $\Delta(1,0,24,0)=-55296\neq 0$, Corollary \ref{item18} states that the number of rational points on $\Ja$ over $\F_{67^n}$ is given by
		$$N_n(\Ja)=\begin{cases}
		67^n-2\left(i\sqrt{67}\right)^n,&\text{ if }n\text{ is even};\\
		67^n,&\text{ if }n\text{ is odd}.\\
		\end{cases}$$
	\end{example}
	
	In Theorem \ref{item16} and Corollary \ref{item18}, we give the number of rational points on $y^2=f(x)$, where $f(x)$ is a polynomial of degree $4$ that is reducible over $\Fq$. In Theorem \ref{item30} we give this number in the case where $f(x)=ax^4+bx^2+c$. Hence, we presented the number of rational points on most hyperelliptic curves of degree $4$ over finite fields of odd characteristic. The missing case is the curve given by the equation $y^2=ax^4+bx^2+cx+d$, with $c\neq 0$, where $f(x)=ax^4+bx^2+cx+d$ is an irreducible polynomial over $\Fq$. In fact, we do not know how to calculate the number of rational points in this specific case.

	\section{Rational Points on Curves of the Form $y^3=f(x)$}\label{sec5}
	In this section, we assume $q\equiv 1\pmod{3}$. The case $q\not\equiv 1\pmod{3}$ is not interesting, since the function $y\mapsto y^3$ permutes the elements of $\Fq$.
	\begin{theorem}\label{item22}
		Let $a,A,B,C\in\Fq$ with $A\neq 0$. Let $\Ca_1: y^3=(x+a)(Ax^2+Bx+C)$ and $\Ca_2: y^3=(x+a)^2(Ax^2+Bx+C)^2$ be curves over $\overline{\Fq}$. If $Aa^2-Ba+C\neq 0$, then the number of rational points on $\Ca_1$ and $\Ca_2$ over $\Fqn$ satisfies
		$$N_n(\Ca_2)+\chi_3\left(A\right)+\chi_3^2\left(A\right)=N_n(\Ca_1)=q^n+1-\omega_q(a',0,0,c')^n-\overline{\omega_q(a',0,0,c')}^{\, n},$$
		where $a':=4Aa^2+4C-4Ba$ and $c':=B^2-4AC$.
	\end{theorem}

		\begin{proof} Since the number of rational points on the $\Ca_1: y^3=(x+a)(Ax^2+Bx+C)$ is equal to the number of rational points on the curve $\Ca': y^3=(x+a)^4(Ax^2+Bx+C)$, unless $3$ possible points at the infinity, the value $N_n(\Ca_1)$ follows from Theorem \ref{item36} and Lemma \ref{item10}. In order to calculate $N_n(\Ca_2)$, we note that
		$$\begin{aligned}N_n(\Ca_2)
		&=1+\sum\limits_{x\in\Fqn}\mkern-6mu\left[1+\chi_3\left((x+a)^2(Ax^2+Bx+C)^2\right)+\chi_3^2\left((x+a)^2(Ax^2+Bx+C)^2\right)\right]\\
		&=1+\sum\limits_{x\in\Fqn}\mkern-6mu\left[1+\chi_3^2\left((x+a)^4(Ax^2+Bx+C)^4\right)+\chi_3\left((x+a)(Ax^2+Bx+C)\right)\right]\\
		&=1+\sum\limits_{x\in\Fqn}\mkern-6mu\left[1+\chi_3^2\left((x+a)(Ax^2+Bx+C)\right)+\chi_3\left((x+a)(Ax^2+Bx+C)\right)\right]\\
		&=N_n(\Ca_1)-\chi_3\left(A\right)-\chi_3^2\left(A\right).\\
		\end{aligned}$$	$\hfill\qed$
		\end{proof}

	\begin{example}
		Let $\Ja_1: y^3=(x+3)(-x^2+2x+2)$ be a curve over $\F_{37}$. In order to calculate $\omega_{37}(-52,0,0,12)$, we use Lemma \ref{item12} as in Example \ref{item39}. We note that
		$$\omega_{37}(-52,0,0,12)+\overline{\omega_{37}(-52,0,0,12)}\equiv\mkern-6mu\sum_{l=6}^{9}\mkern-4mu \binom{18}{2l}\mkern-4mu\binom{2l}{18-l}\mkern-3mu(-52)^{18-l}0^{3l-18}12^{18-2l}\!\equiv 27\mkern-12mu\pmod{37},$$
		then
		$\omega_{37}(-52,0,0,12)=-5+i\sqrt{12} $.  Theorem \ref{item22} states that the number of rational points on $\Ja_1$ over $\F_{37^n}$ is given by
		$$N_n(\Ja_1)=37^n+1-\left(-5+i\sqrt{12}\right)^n-\left(-5-i\sqrt{12}\right)^n.$$
		In addition, the number of rational points on the curve $\Ja_2:y^3=(x+3)^2(-x^2+2x+2)^2$ over $\F_{37^n}$ is given by
		$$N_n(\Ja_2)=37^n-1-\left(-5+i\sqrt{12}\right)^n-\left(-5-i\sqrt{12}\right)^n.$$
	\end{example}

	\begin{theorem}\label{item40}
		Let $a,b\in\Fq^*$. The number of rational points on $\Ca: y^3=ax^3+b$ over $\Fqn$ satisfies
		$$N_n(\Ca)=q^n+1-\omega_q(a^{-1},0,0,(b/2a)^2)^n-\overline{\omega_q(a^{-1},0,0,(b/2a)^2)}^{\, n}.$$
	\end{theorem}

	\begin{proof}We have
		$$\begin{aligned}N_n(\Ca)&=1+\chi_3(a)+\chi_3^2(a)+\sum\limits_{x\in\Fqn}\left[1+\chi_3(ax+b)+\chi_3^2(ax+b)\right]\left[1+\chi_3(x)+\chi_3^2(x)\right]\\
		&=1+\chi_3(a)+\chi_3^2(a)+S_1+S_2,
		\end{aligned}$$
		where
		$$\begin{aligned}S_1:&=\sum\limits_{x\in\Fqn}\left[1+\chi_3(ax^2+bx)+\chi_3^2(ax^2+bx)\right]\\
		&=\sum\limits_{z\in\Fqn}\left[1+\chi_3\left(a\left(z-\tfrac{b}{2a}\right)^2+b\left(z-\tfrac{b}{2a}\right)\right)+\chi_3^2\left(a\left(z-\tfrac{b}{2a}\right)^2+b\left(z-\tfrac{b}{2a}\right)\right)\right]\\
		&=\sum\limits_{z\in\Fqn}\left[1+\chi_3\left(az^2-\tfrac{b^2}{4a}\right)+\chi_3^2\left(az^2-\tfrac{b^2}{4a}\right)\right]\\
		\end{aligned}$$
		and
		$$\begin{aligned}S_1:&=\sum\limits_{x\in\Fqn^*}\left[\chi_3(ax+b)\chi_3^2(x)+\chi_3^2(ax+b)\chi_3^4(x)\right]\\
		&=\sum\limits_{x\in\Fqn^*}\left[\chi_3(a+b\tfrac{1}{x})+\chi_3^2(a+b\tfrac{1}{x})\right]\\
		&=\sum\limits_{z\in\Fqn^*}\left[\chi_3(a+bz)+\chi_3^2(a+bz)\right]\\
		&=-\chi_3(a)-\chi_3^2(a).
		\end{aligned}$$
		Since $S_1$ count the number of rational points in the curve with equation $z^2=\tfrac{y^3}{a}+\tfrac{b^2}{4a^2}$, the result follows from Lemma \ref{item10}.$\hfill\qed$\\
	\end{proof}

	\begin{proof} (proof of Theorem \ref{item23})
		We note that
		$$\begin{aligned}
		N_n(\Ca)&=1+\sum\limits_{x\in\Fqn}\big[1+\chi_3(ax^2+b)+\chi_3^2(ax^2+b)\big]\big[1+\chi_3(x)+\chi_3^2(x)\big]\\
		&=1+S_1+S_2+S_3,
		\end{aligned}$$
		where 
		$$\begin{aligned}S_1&:=\sum\limits_{x\in\Fqn}\big[1+\chi_3(ax^2+b)+\chi_3^2(ax^2+b)+\chi_3(x)+\chi_3^2(x)\big]\\
		&=\sum\limits_{x\in\Fqn}\big[1+\chi_3(ax^2+b)+\chi_3^2(ax^2+b)\big]\\
		&=\sum\limits_{w\in\Fqn}\big[1+\chi_2(a^{-1}w^3-ba^{-1})\big],\\
		\end{aligned}$$
		$$\begin{aligned}S_2&:=\sum\limits_{x\in\Fqn}\big[\chi_3(ax^2+b)\chi_3(x)+\chi_3^2(ax^2+b)\chi_3^2(x)\big]\\
		&=\sum\limits_{x\in\Fqn^*}\left[\chi_3\left(\frac{ax^2+b}{x^2}\right)+\chi_3^2\left(\frac{ax^2+b}{x^2}\right)\right]\\
		&=\sum\limits_{z\in\Fqn^*}\left[\chi_3(a+bz^2)+\chi_3^2(a+bz^2)\right]\\
		&=\sum\limits_{w\in\Fqn}\chi_2(b^{-1}w^3-ab^{-1})-\chi_3(a)-\chi_3^2(a)\\
		\end{aligned}$$
		and 
		$$\begin{aligned}S_3&:=\sum\limits_{x\in\Fqn}\big[\chi_3(ax^2+b)\chi_3^2(x)+\chi_3^2(ax^2+b)\chi_3(x)\big]\\
		&=\sum\limits_{x\in\Fqn^*}\left[\chi_3\left(\frac{ax^2+b}{x}\right)+\chi_3^2\left(\frac{ax^2+b}{x}\right)\right].\\
		\end{aligned}$$
		By Lemma \ref{item10},
		$$S_1=q^n-\omega_1^n-\overline{\omega_1}^n,$$
		 where $\omega_1:=\omega_q(a^{-1},0,0,-ba^{-1})$ and
		$$S_2=-\omega_2^n-\overline{\omega_2}^n-\chi_3(a)-\chi_3^2(a),$$
		where $\omega_2:=\omega_q(b^{-1},0,0,-ab^{-1})$.	By Remark \ref{item37} and Theorem \ref{item17},
		$$S_3=-\omega_3^n-\overline{\omega_3}^n-\omega_4^n-\overline{\omega_4}^n,$$
		where $\omega_3:=\omega_q(1,0,0,-4ab)$ and $\omega_4:=\omega_q(-4ab,0,0,1)$. $\hfill\qed$
		
	\end{proof}

	\begin{example}
		Let $\Ja: y^3=x^6+1$ be a curve over $\F_{103}$. We use Lemma \ref{item12} as in Example \ref{item39} to calculate the complex numbers $\omega_i$ in Theorem \ref{item23}. We have
		$$\omega_{103}(1,0,0,-1)=-10+i\sqrt{3},\ \omega_{103}(1,0,0,-4)=\tfrac{7+i\sqrt{363}}{2}\text{ and } \omega_{103}(-4,0,0,1)=\tfrac{-13+i\sqrt{243}}{2}.$$
		Then, by Theorem \ref{item23},
		$$N_n(\Ja)=103^n-1-2\cdot\omega_1^n-2\cdot\overline{\omega_1}^n-\omega_2^n-\overline{\omega_2}^n-\omega_3^n-\overline{\omega_3}^n,$$
		where $\omega_1:=-10+i\sqrt{3},\ \omega_2:=\tfrac{7+i\sqrt{363}}{2}$ and $\omega_3:=\tfrac{-13+i\sqrt{243}}{2}$.
	\end{example}

\section{Rational Points on Curves of the Form $y^4=f(x)$}\label{sec6}

In this section, we calculate the number of rational points on curves of the form $y^4=ax^4+bx^2+c$. The case $q\equiv 3\pmod{4}$ must be considered separately, since every square is a fourth power, as we show in the following lemma.

\begin{lemma}\label{item44}
	Let $\Fq$ be a finite field with $q\equiv 3\pmod{4}$ elements and $k$ a positive integer. An element $\alpha \in\Fq$ is a square if and only if it is a $2^k$ power.
\end{lemma}

\begin{proof} Since $\gcd(2^{k-1},\tfrac{q-1}{2})=1$, there are integers $a,b$ such that $2^{k-1}\cdot a+\tfrac{q-1}{2}\cdot b=1$. Then
	$$\alpha^{2}=\alpha^{2^k\cdot a+(q-1)\cdot b}=(\alpha^a)^{2^k}$$
	for all $\alpha\in\Fq$. Conversely,
	$$\alpha^{2^k}=\alpha^{2^{2k-1}a+2^{k-1}(q-1)b}=\left(\alpha^{2^{2k-2}a}\right)^2.$$ $\hfill\qed$
\end{proof}

Therefore, in the case $q\equiv 3\pmod{4}$, the number of rational points on $\mathcal{L}:y^4=ax^4+bx^2+c$ is the same as the number of rational points on $y^2=ax^4+bx^2+c$ unless points at the infinity. We have already presented the number of points on $y^2=ax^4+bx^2+c$ in Theorem \ref{item30}. In order to calculate $N_n(\mathcal{L})$ for any positive integer $n$ in the case where the $a,b,c$ are elements in a prime field, we have to calculate the number of points in these curves in the case $q\equiv 1\pmod{4}$.

\begin{lemma}\label{item20}Let $a,b,c\in\Fq$, with $a\neq 0$ and $q\equiv 1\pmod{4}$. For $f(x):=ax^2+bx+c$, we have
	$$\sum\limits_{x\in\Fqn}\big[\chi_4(f(x))+\chi_4^3(f(x))\big]=\begin{cases}
	-\omega_q(a^{-1},0,d,0)^n-\overline{\omega_q(a^{-1},0,d,0)}^{\, n}\!\!\!,&\text{ if }b^2-4ac\neq 0;\\
	0,&\text{ if }b^2-4ac= 0,\\
	\end{cases}$$
	where $d:=\tfrac{b^2-4ac}{4a^2}$.
\end{lemma}

\begin{proof}We observe that 
	$$\begin{aligned}&1+\sum\limits_{x\in\Fqn} \big[1+\chi_4(ax^2+bx+c)+\chi_4^2(ax^2+bx+c)+\chi_4^3(ax^2+bx+c)\big]\\
	&=1+\sum\limits_{(z-b/2a)\in\Fqn} \big[1+\chi_4\big(az^2-\tfrac{b^2-4ac}{4a}\big)+\chi_4^2\big(az^2-\tfrac{b^2-4ac}{4a}\big)+\chi_4^3\big(az^2-\tfrac{b^2-4ac}{4a}\big)\big]
	\end{aligned}$$
	calculate the number of rational points on the curve $\Ca: z^2=\tfrac{y^4}{a}+\tfrac{b^2-4ac}{4a^2}$. 
	Therefore, letting $L:= \sum\limits_{x\in\Fqn}\big[\chi_4(f(x))+\chi_4^3(f(x))\big]$, by Theorem \ref{item30}, we have
	$$L=-\!\sum\limits_{x\in\Fqn} \chi_4^2(ax^2+bx+c)+\begin{cases}
	-\omega_q(a^{-1},0,d,0)^n-\overline{\omega_q(a^{-1},0,d,0)}^{\, n}-\chi_2(a),&\!\!\!\text{ if }b^2-4ac\neq 0;\\
	(q^n-1)\chi_2(a),&\!\!\!\text{ if }b^2-4ac= 0,\\
	\end{cases}$$
	where $d:=\tfrac{b^2-4ac}{4a^2}$. The result follows by Lemma \ref{item25}.$\hfill\qed$\\
\end{proof}

\begin{proof} (proof of Theorem \ref{item19})
	We have
	$$\begin{aligned}N_n(\Ca)\!&=\!\delta\!+\!\!\sum\limits_{x\in\Fqn}\!\big[1\!+\!\chi_4(ax^2\!+bx\!+c)\!+\!\chi_4^2(ax^2\!+bx\!+c)\!+\!\chi_4^3(ax^2\!+bx\!+c)\big]\big[1\!+\!\chi_4^2(x)\big]\\
	&= \delta+S_1+S_2,\\
	\end{aligned}$$
	where $\delta:=1+\chi_4(a)+\chi_4^2(a)+\chi_4^3(a)$.
	$$S_1:=\sum\limits_{x\in\Fqn} 1+\chi_4(ax^2+bx+c)+\chi_4^2(ax^2+bx+c)+\chi_4^3(ax^2+bx+c)$$
	and  
	$$S_2:=\sum\limits_{x\in\Fqn^*} \big[1+\chi_4(ax^2+bx+c)+\chi_4^2(ax^2+bx+c)+\chi_4^3(ax^2+bx+c)\big]\chi_4^2(x).$$
	By Remark \ref{item32},
	$$\begin{aligned}S_2
	& =\sum\limits_{x\in\Fqn^*}\!\left[\chi_4^2(x)\!+\chi_4\!\left(\!\frac{ax^2\!+bx\!+c}{x^2}\right)\!+\chi_4^2(ax^3\!+bx^2\!+cx)\!+\chi_4^3\!\left(\!\frac{ax^2+bx+c}{x^2}\right)\right]\\
	& =\sum\limits_{x\in\Fqn^*}\chi_4^2(ax^3+bx^2+cx)+\sum\limits_{z\in\Fq^*}\big[ \chi_4(a+bz+cz^2)+\chi_4^3(a+bz+cz^2)\big].\\
	\end{aligned}$$
	Using Lemma \ref{item25} and Lemma \ref{item20} in $S_1$ and  Lemma \ref{item11} (in the case where $c=0$), Lemma \ref{item10} and Lemma \ref{item20} in $S_2$, the result follows.$\hfill\qed$
\end{proof}

\begin{example}
	Let $\Ja: y^4=x^4+4x^2-1$ be a curve over $\F_{41}$. We use Lemma \ref{item12} as in Example \ref{item39} to calculate the complex numbers $\omega_i$ in Theorem \ref{item19}. We have
	$$\omega_{41}(1,0,5,0)=\omega_{41}(-1,0,5,0)=\omega_{41}(1,4,-1,0)=-5+4 i.$$
	Then, by Theorem \ref{item19}, the number of rational points on $\Ja$ over $\F_{41^n}$ is given by
	$$N_n(\Ja)=41^n+1-3\cdot(-5+4 i)^n-3\cdot(-5-4 i)^n.$$
\end{example}

In the following result, we use that the number of rational points on $\mathcal{L}:y^4=ax^4+bx^2+c$ is essentially the number of points on $\mathcal{L}':y^2=ax^4+bx^2+c$, as we have seen in Lemma \ref{item44}.

\begin{corollary}\label{item43}
	For $a,b,c\in\Fp$, where $p\equiv 3\pmod{4}$ is a prime number and $a\neq0$. The number of rational points on the curve $\Ca: y^4=ax^4+bx^2+c$ over $\Fpn$ satisfies
	$$N_n(\Ca)\!=\!\begin{cases}
	p^n+1-2(i\sqrt{p})^n-2(-i\sqrt{p})^n-\omega^n-\overline{\omega}^{\, n},&\text{if }b^2-4ac\neq 0\text{ and }c\neq0;\\
	p^n+1-(i\sqrt{p})^n-(-i\sqrt{p})^n-\chi_2(b),&\text{if }b^2-4ac\neq 0\text{ and }c=0;\\
	p^n+1-\chi_2(-b/2)+p^n\cdot\chi_2(a),&\text{if }b^2-4ac=0\text{ and }c\neq0;\\
	p^n+1-\chi_2(b)+p^n\cdot\chi_2(a),&\text{if }b^2-4ac=0\text{ and }c=0,\\
	\end{cases}$$
	where $\omega:=\omega_p(a,b,c,0)$.
\end{corollary}
\begin{proof}By Theorem \ref{item19} and Remark \ref{item45},
$$N_{2n}(\Ca)\!=\!\begin{cases}
p^{2n}+1-\omega_1^{2n}-\overline{\omega_1}^{\, 2n}-\omega_2^n-\overline{\omega_2}^{\, 2n}-\omega_3^{2n}-\overline{\omega_3}^{\, 2n},&\text{if }b^2-4ac\neq 0\text{ and }c\neq0;\\
p^{2n}+1-\omega_1^{2n}-\overline{\omega_1}^{\, 2n}-\chi_2(b),&\text{if }b^2-4ac\neq 0\text{ and }c=0;\\
p^{2n}+1-\chi_2(-b/2)+p^{2n}\cdot\chi_2(a),&\text{if }b^2-4ac=0\text{ and }c\neq0;\\
p^{2n}+1-\chi_2(b)+p^{2n}\cdot\chi_2(a),&\text{if }b^2-4ac=0\text{ and }c=0,\\
\end{cases}$$
where $\omega_1:=\omega_p(a^{-1},0,d_1,0)$, $d_1:=\tfrac{b^2-4ac}{4a^2}$, $\omega_2:=\omega_p(c^{-1},0,d_2,0)$, $d_2:=\tfrac{b^2-4ac}{4c^2}$ and $ \omega_3:=\omega_p(a,b,c,0)$. Lemma \ref{item12} states that
$$\omega_1+\overline{\omega_1}\equiv\omega_2+\overline{\omega_2}\equiv 0\pmod{p}.$$
Since $|\omega_1+\overline{\omega_1}|\leq 2\sqrt{p}$ and $|\omega_2+\overline{\omega_2}|\leq2\sqrt{p}$, we must have 
$$\omega_1=\omega_2=0$$
for $p>4$. A straightforward calculation shows that $\omega_3(a^{-1},0,d_1,0)=\omega_3(c^{-1},0,d_2,0)=0$. In addition, since there are $1+\chi_2(a)$ points at the infinity on $\Ca$ (by Lemma \ref{item44}, each element $a\in\Fp$ has exactly $0$ or $2$ fourth roots), by Theorem \ref{item30}, we have
$$N_{2n-1}(\Ca)=\begin{cases}
p^{2n-1}+1-\omega_3^{2n-1}-\overline{\omega_3}^{\, 2n-1},&\text{ if }b^2-4ac\neq 0\text{, with }c\neq0;\\
p^{2n-1}+1-\chi_2(b),&\text{ if }b^2-4ac\neq 0\text{, with }c=0;\\
p^{2n-1}+1-\chi_2(-b/2)+p^{2n-1}\cdot\chi_2(a),&\text{ if }b^2-4ac= 0\text{, with }c\neq0;\\
p^{2n-1}+1+p^{2n-1}\cdot\chi_2(a),&\text{ if }b^2-4ac= 0\text{, with }c=0\\
\end{cases}$$
for all positive integer $n$. The result follows by gathering the expressions for $N_{2n}$ and $N_{2n-1}$.$\hfill\qed$
\end{proof}

\section{Maximal and Minimal Curves}\label{sec7}

Let $\F_{q^2}$ be a fields with $q^2$ elements. Let $\Ca$ be a projective, geometrically irreducible and non-singular algebraic curve defined over $\F_{q^2}$. The curve $\Ca$ is called maximal over $\F_{q^2}$ if it attains the Hasse-Weil upper bound, that is,
$$N_1(\Ca)=q^2+1+2gq,$$
where $g$ is the genus of $\Ca$. Similarly, a curve is called minimal over $\F_{q^2}$ if it attains the Hasse-Weil lower bound. Many researchers have studied maximal and minimal curves, e.g. see \cite{fuhrmann1997maximal,giulietti2009new,kazemifard2013note,tafazolian2017note}. In this section, we use the theorems presented in the previous sections in order to produce families of maximal and minimal curves. We will consider the curves with equation 
$$ax^n+by^n+cz^n=0$$
in the cases $n=3$ and $n=4$, where $a,b,c$ are elements in a prime field. The special case where $a=b=c=1$, that is, the projective curves with equation
$$x^n+y^n+z^n=0,$$
are well-known as Fermat curves. A general criteria to decide when a Fermat curve $\Ca:x^n+y^n+z^n=0$ is maximal or minimal was presented by Garcia and Tafazolian  \cite{garcia2008cartier}. A shorter proof for the same result was presented by Tafazolian \cite{tafazolian2010characterization}.

\begin{lemma}[{\cite{Lucas}, Lucas's Theorem}]\label{item21} For non-negative integers $m$ and $n$ and a prime $p$, the following holds:
	$$ \binom{n}{m} \equiv \binom{n_0}{m_0}\binom{n_1}{m_1}\dots\binom{n_k}{m_k}\ \pmod{p},$$
	where $\binom{n_i}{m_i}:=0$ if $m_i>n_i$ and $m=m_kp^k+\dots+m_1 p+m_0$, $n=n_kp^k+\dots+n_1 p+n_0$ are the base $p$ expansions of $m$ and $n$. 
\end{lemma}

\begin{theorem}\label{item48}
	Let $a,b,c$ be non-null elements in a prime field $\Fp$, where $p\neq3$. Let $\Ca$ be the projective non-singular curve of genus $1$ given by equation $ax^3+by^3+cz^3=0$. Then
	\begin{enumerate}[i)]
		\item $\Ca$ is maximal over $\F_{p^{2n}}$ if and only if $3$ divides $p^n+1$;
		\item $\Ca$ is minimal over $\F_{p^{2n}}$ if and only if $3$ divides $p^n-1$ and $p\equiv 2\pmod{3}$.
	\end{enumerate}
\end{theorem}

\begin{proof}
	We note that the $\Ca$ has the same number of rational points as the affine curve $\Ca'$ given by equation $y^3=-\tfrac{a}{b}x^3-\tfrac{c}{b}$. Since $\Ca'$ has genus $1$, the number of rational points on $\Ca'$ is given by
	$$N_n(\Ca')=p^n+1-\omega^n-\overline{\omega}^n,$$
	for a complex number $\omega$, where $|\omega|=\sqrt{p}$. In the case $p\equiv 2\pmod{3}$, since $y\mapsto y^3$ permutes the elements of $\Fp$, we have $N_1(\Ca')=p+1$ and then $\omega=i\sqrt{p}$. Hence, the result is proved in this case. For $p\equiv 1\pmod{3}$, Theorem \ref{item40} and Lemma \ref{item12} show that
	$$N_n(\Ca')\equiv 1-\binom{\tfrac{p^n-1}{2}}{\tfrac{p^n-1}{3}}\left(\frac{-b}{a}\right)^{\frac{p^n-1}{3}}\left(\frac{c^2}{4a^2}\right)^{\frac{p^n-1}{6}}\pmod{p}.$$
	By Lemma \ref{item21}, we must have
	$$\binom{\tfrac{p^n-1}{2}}{\tfrac{p^n-1}{3}}\equiv \binom{\tfrac{p-1}{2}}{\tfrac{p-1}{3}}^n\not\equiv 0\pmod{p}.$$
	Thus $N_n(\Ca')\neq p^n+1$ and then $\omega\neq i\sqrt{p^n}$. Hence, the result is shown.
	$\hfill\qed$
\end{proof}

\begin{theorem}\label{item46}
	Let $a,b,c$ be non-null elements in a prime field $\Fp$. Let $\Ca$ be the projective non-singular curve of genus $3$ given by equation $ax^4+by^4+cz^4=0$. Then
	\begin{enumerate}[i)]
		\item $\Ca$ is maximal over $\F_{p^{2n}}$ if and only if $4$ divides $p^n+1$;
		\item $\Ca$ is minimal over $\F_{p^{2n}}$ if and only if $4$ divides $p^n-1$ and $p\equiv 3\pmod{4}$.
	\end{enumerate}
\end{theorem}

\begin{proof}
	Since the projective curve with equation $ax^4+by^4+cz^4=0$ has the same number of rational points as the affine curve given by equation $y^4=-\tfrac{a}{b}x^4-\tfrac{c}{b}$, the result follows from Corollary \ref{item43} and Lemma \ref{item12} for the case $p\equiv 3\pmod{4}$. For the case $p\equiv 1\pmod{4}$, Lemma \ref{item12} states that
	$$\omega_{p^n}(-a/b,0,-c/b,0)\equiv \binom{\tfrac{p^n-1}{2}}{\tfrac{p^n-1}{4}}\left(\frac{-b}{a}\right)^{\frac{p^n-1}{4}}\left(\frac{c^2}{4a^2}\right)^{\frac{p^n-1}{4}}\pmod{p}.$$
	Since 
	$$\binom{\tfrac{p^n-1}{2}}{\tfrac{p^n-1}{4}}\equiv \binom{\tfrac{p-1}{2}}{\tfrac{p-1}{4}}^n\not\equiv 0\pmod{p}$$
	by Lemma \ref{item21}, the result follows from Theorem \ref{item19}.$\hfill\qed$
\end{proof}
 The results proved in \ref{item46} and \ref{item48} generalize the result of Tafazolian \cite{tafazolian2010characterization} for Fermat curves of degree $3$ and $4$.
 \begin{conjecture}
 	The prime number $p$ in Theorems \ref{item46} and \ref{item48} can be replaced by a power of a prime.
 \end{conjecture}
From the results presented in this section, it emerge the following question.
 \begin{question}
	Are the conditions for the maximality and minimality of the curve $ax^n+by^n+cz^n=0$ the same as the conditions for the Fermat curve $x^n+y^n+z^n=0$ presented in \cite{tafazolian2010characterization} for all positive integer $n$?
\end{question}
In the same way of the proofs of Theorems \ref{item48} and \ref{item46}, it is possible give conditions for which the singular curves presented in this paper are optimal. But, none of the singular curves presented here attains the upper bound 
$$q^n+1+2g\sqrt{q}+\pi-g$$
in Theorem \ref{item41}.

\section{Acknowledgments}

I am grateful to Fabio Enrique Brochero Martínez for his useful comments. This study was financed in part by the Coordenação de Aperfeiçoamento de Pessoal de Nível Superior - Brasil (CAPES) - Finance Code 001.

\bibliographystyle{siam}
\bibliography{biblio}
	
\end{document}